\documentclass{amsart}
\usepackage[utf8]{inputenc}
\usepackage{color}

\usepackage[margin=1.0in]{geometry}

\usepackage{amsmath}
\usepackage{amsthm}
\usepackage{amssymb}
\usepackage[matrix,arrow,curve]{xy}
\usepackage{enumitem} 
\usepackage{graphicx}
\usepackage{mdwlist}	
\usepackage{mathrsfs}
\usepackage{tikz-cd}
\usetikzlibrary{patterns}
\usepackage{caption}
\usepackage{subcaption}
\usepackage{makecell}
\usepackage{ textcomp }
\usepackage{longtable}
\usepackage{mathdots}

\makeatletter
\def\iddots{\mathinner{\mkern1mu\raise\p@
\vbox{\kern7\p@\hbox{.}}\mkern2mu
\raise4\p@\hbox{.}\mkern2mu\raise7\p@\hbox{.}\mkern1mu}}
\makeatother

\newcommand\setItemnumber[1]{\setcounter{enum\romannumeral\@enumdepth}{\numexpr#1-1\relax}}

\newcommand{\p}{\mathbb{P}}
\DeclareMathOperator{\Bir}{Bir}
\DeclareMathOperator{\Aut}{Aut}

\DeclareMathOperator{\PGL}{PGL}

\DeclareMathOperator{\ZZ}{\mathbb{Z}}
\DeclareMathOperator{\RR}{\mathbb{R}}
\DeclareMathOperator{\CC}{\mathbb{C}}

\renewcommand{\AA}{\mathbb{A}}
\DeclareMathOperator{\Pic}{Pic}
\DeclareMathOperator{\NS}{NS}

\DeclareMathOperator{\pr}{pr}

\DeclareMathOperator{\Exc}{Exc}

\DeclareMathOperator{\Ind}{Ind}

\DeclareMathOperator{\Alb}{Alb}

\DeclareMathOperator{\OP}{\mathcal{O}}

\renewcommand{\k}{\mathbb{K}}
\renewcommand{\L}{\mathcal{L}}

\newcommand{\F}{\mathcal{F}}

\newcommand{\M}{\mathcal{M}}

\newcommand{\xdashrightarrow}[2][]{\ext@arrow 0359\rightarrowfill@@{#1}{#2}}

\newtheorem{theorem}[equation]{Theorem}

\newtheorem{lemma}[equation]{Lemma}

\newtheorem{corollary}[equation]{Corollary}

\theoremstyle{definition}

\makeatletter
\@addtoreset{equation}{section}
\makeatother

\title{Automorphisms of bounded growth}
\author{Alexandra Kuznetsova}

\address{Higher School of Modern Mathematics in Moscow Institute of Physics and Technology, Dolgoprudny, Russia; Laboratory of Algebraic Geometry in Higher School of Economics, Moscow, Russia;  Steklov Mathematical Institute of Russian Academy of Sciences, Moscow, Russia.}

 \email{sasha.kuznetsova.57@gmail.com}

\begin{document}

 \begin{abstract}
   We study birational automorphisms of algebraic varieties of bounded growth, i.e. such that the  norms of the inverse images $(f^n)^*\colon \mathrm{NS}(X)\to\mathrm{\NS(X)}$ of the powers of the automorphism $f\in\Bir(X)$ are bounded above for $n\geqslant 0$. We prove that some power of an infinite order automorphism of a variety $X$ with such property factors either through an infinite order translation on the Albanese variety of~$X$ or through an infinite order regular automorphism of~$\p^m$ for~\mbox{$m\geqslant 1$}. We deduce from this that if a rationally connected threefold admits an infinite order automorphism whose growth is bounded then the threefold is rational and an iterate of the automorphism is birationally conjugate to a regular automorphism of $\p^3$, a generalization of Blanc and Deserti's result \cite{Blanc-Deserti}.
 \end{abstract}
\maketitle

\section{Introduction}
Let $X$ be a projective variety over an algebraically closed field $\k$ of characteristic 0 and let $f$ be a birational automorphism of $X$. Fix an ample class $H$ in the Neron--Severi group $\NS(X)$. Then we can consider the sequence of numbers $((f^n)^*H) \cdot H^{\dim(X) - 1}$. If the automorphism $f$ is birationally conjugate to a regular automorphism of a projective variety then the asymptotic properties of this sequence do not depend on the choice of the ample class $H$ and the birational model of $X$ and its growth can be 
 \begin{enumerate}
  \item[(1)] exponential, i.e. the limit $\lim_{n\to \infty}\left(((f^n)^*H) \cdot H^{\dim(X) - 1}\right)^{\frac{1}{n}}$ exists and is strictly greater than $1$;  
  \item[(2)] polynomial, i.e. $((f^n)^*H) \cdot H^{\dim(X) - 1} = O(n^k)$ for an integer $k\geqslant 1$;
  \item[(3)] bounded, i.e. $((f^n)^*H) \cdot H^{\dim(X) - 1} \leqslant C$ for a real number $C\geqslant 0$.
 \end{enumerate}
 The limit in the point \textup{(1)} exists and is birational invariant for not necessarily regularizable automorphisms of a variety $X$ of any dimension, see~\cite{DS_Lambdas} and \cite{TTT_dynamical_degrees}. If the growth of $f$ is polynomial then there exist bounds on the integer $k$, see \cite{DLOZ} and \cite{HJ}. In the case of $\dim(X) = 2$ this classification holds also in the case of not necessarily regularizable birational automorphisms by \cite{Diller_Favre}. 
 
 Here we are going to study birational automorphisms of the third type i.e. those whose growth is bounded. It is easy to observe that the growth of any  finite order automorphism is bounded. We focus on birational automorphisms of $X$ whose order is infinite and the growth is bounded. Any birational automorphism with such property is necessarily  regularizable on a projective birational model of~$X$ by~\cite{Weil}, see also \cite{Kraft} and \cite{Yang}; thus, it suffices to study regular automorphisms with these properties. Our goal is to show that they are sufficiently rare. 
 
 First, let us consider the situation in low dimensions. If $X$ is a curve and $f$ is its automorphism, then its growth is necessarily bounded and the infinite order automorphisms of curves exist either if $X\cong\p^1$ and~$f$ is an infinite order element of $\Aut(\p^1)\cong\PGL(2,\k)$ or if $X$ is an elliptic curve and an iterate of $f$ is its translation. 
 
 In the case when $X$ is a surface it is easy to show that $f$ is birationally conjugate to a regular automorphism of a bielliptic surface or an abelian surface or a ruled surface.  Once $f$ is a regular automorphism of a bielliptic or an     abelian surface then an iterate of $f$ is induced by a translation of an abelian surface (in the case of a bielliptic surface it is a translation on the universal cover of the surface).  Once $X$ is birational to $\p^1\times C$ where the genus of the curve $C$ equals at least 1 then using the minimal model program one can show that the automorphism $f$ is induced by a product $f_{\p^1}\times f_{C}\in \Aut(\p^1)\times \Aut(C)\subset \Aut(\p^1\times C)$ and either the order of~\mbox{$f_{\p^1}\in \Aut(\p^1)$} is infinite or $C$ is an elliptic curve and the order of $f_C$ is infinite. The most difficult case when $S$ is rational is addressed in the work by Blanc and Deserti; more precisely, they proved the following assertion:
 \begin{theorem}[{\cite[Theorem A]{Blanc-Deserti}}]\label{theorem: Blanc-Deserti}
  Let $S$ be a rational surface over an algebraically closed field of characteristic $0$ and let $f$ be its infinite order automorphism whose growth is bounded. Then there exists a birational map $\alpha\colon S\dashrightarrow \p^2$ such that $\alpha\circ f\circ \alpha^{-1}$ is a regular automorphism of $\p^2$.
 \end{theorem}
 Thus, in the case of dimensions 1 and 2 the only sources of automorphisms whose orders are infinite and growths are bounded are translations of abelian varieties and regular automorphisms of projective spaces. We show that in higher dimensions the situation is similar. Recall that $\Alb(X)$ is the Albanese variety of~$X$. Here is the main result of this paper.
\begin{theorem}\label{theorem: automorphism of bounded growth implies ruled structure or translation on albanese}
 Let $X$ be a smooth projective complex variety and let $f \in \Aut(X)$ be an automorphism whose growth is bounded. Then we have one of the following cases:
 \begin{enumerate}
  \item[\textup{(1)}] the order of $f$ is finite;
  \item[\textup{(2)}] there exists $N\geqslant 1$ such that $f^N$ induces an infinite order translation on $\Alb(X)$;
  \item[\textup{(3)}] there exists a birational isomorphism $\alpha\colon X\dashrightarrow \p^m\times Z$ for a smooth projective variety $Z$ with $m\geqslant 1$ such that $\alpha\circ f\circ \alpha^{-1} = g\times h\in \Aut(\p^m)\times \Aut(Z)$ where 
 \begin{enumerate}
  \item[\textup{(a)}] the order of the automorphism $g\in\Aut(\p^m)$ is infinite;
  \item[\textup{(b)}] the order of the automorphism $h\in\Aut(Z)$ is finite.
 \end{enumerate}
 \end{enumerate}
\end{theorem}
 Note that in the case (2) in Theorem \ref{theorem: automorphism of bounded growth implies ruled structure or translation on albanese} the automorphism $f^N$ factors through an infinite order translation on the abelian variety whether in the case (3) the automorphism $f^N$ factors through an infinite order regular automorphism of the projective space $\p^m$, here $N$ equals the order of the automorphism $h$.

 Theorem \ref{theorem: automorphism of bounded growth implies ruled structure or translation on albanese} implies the following generalization of Theorem \ref{theorem: Blanc-Deserti}.
\begin{corollary}\label{corollary: generalization of Blanc-Deserti}
 If $X$ is a rationally connected threefold over an algebraically closed field of characteristic~$0$ and $f$ is its infinite order automorphism whose growth is bounded then $X$ is rational. Moreover, there exists~\mbox{$N\geqslant 1$}  and a birational map $\alpha\colon X\dashrightarrow \p^3$ such that $\alpha\circ f^N\circ \alpha^{-1}$ is a regular automorphism of $\p^3$.
\end{corollary}
In higher dimensions the same assertion does not hold. Indeed, fix a stably non-rational Fano threefold~$X$ and an automorphism $g\in \Aut(\p^k)$ such that $k\geqslant 1$ and the order of $g$ is infinite. Then $f=\mathrm{Id}_{X}\times g$ is an automorphism of $X\times \p^k$ whose order is infinite and the growth is bounded while $X\times \p^k$ is a rationally connected non-rational variety of dimension $3+k$.

Also unlike the two-dimensional case we can not expect that if $f$ is an infinite order automorphism of a rationally connected threefold $X$ whose order is bounded then not only $f^N$ but also $f$ is conjugate to a regular automorphism of $\p^3$. See Section \ref{section: example} for a counterexample.

Our approach to Theorem \ref{theorem: automorphism of bounded growth implies ruled structure or translation on albanese} differs from the one used in the Blanc and Deserti's paper. They applied the equivariant minimal model program to reduce the question to the study of infinite order automorphisms of del Pezzo surfaces and conic bundles. In higher dimensions this method is very inefficient since the minimal model program does not exist in the case of dimension 4 and higher. Also even in the case of dimension 3 it would reduce the question to the study of regular automorphisms of terminal Fano threefolds and three dimensional Mori fiber spaces. There are methods to work with these varieties; however, this approach would require a lot of work even in the case of dimension 3. 

Instead of this, we observe that unless the order of the induced automorphism on $\Alb(X)$ is infinite, an automorphism~\mbox{$f\in\Aut(X)$} whose growth is bounded preserves a big and nef line bundle $\L$ on $X$. The linear system of $\L$ defines a rational map from $X$ to $\p(H^0(X,\L))\cong \p^n$ and $X$ is birational to its image $X'$ in~$\p^n$. The automorphism $f$ induces $F\in\Aut(\p^n)$ such that $X'$ is an $F$-invariant subvariety. Using Blanc's result~\cite{Blanc_conjugacy_of_PGL} we deduce that if the order of $f$ is infinite then after a birational conjugation $F$ is induced by a very special matrix. In particular, all $F$-invariant subvarieties  which are not contained in hyperplanes of $\p^n$ are cones with same vertex and $F$ induces an automorphism of finite order on the bases of the cones. Thus, we conclude that $X$ itself is such a cone i.e it is birationally isomorphic to a product $\p^m\times Z$ as we claimed. 

Note that Theorem \ref{theorem: automorphism of bounded growth implies ruled structure or translation on albanese} is proved only for complex algebraic varieties whereas Corollary \ref{corollary: generalization of Blanc-Deserti} holds for varieties over any algebraically closed field of characteristic 0. The reason is that in order to study automorphisms whose inverse image $f^*\colon \Pic(X)\to \Pic(X)$ preserve no big and nef line bundle we use the fact that in the complex setting $\Pic^0(X)$ is isomorphic to the quotient $H^1(X,\OP_X)/H^1(X,\ZZ)$. We need this to deal with the case (2) of Theorem \ref{theorem: automorphism of bounded growth implies ruled structure or translation on albanese}. In Corollary \ref{corollary: generalization of Blanc-Deserti} we consider rationally connected varieties, their Picard groups are discrete; thus, we can skip this step and the proof works in greater generality.

\section{Automorphisms of projective space}\label{section: Blanc theorem and corollaries}
\subsection{Conjugacy classes of $\Aut(\p^n)$ under the action of $\Bir(\p^n)$}
The projective space $\p^n$ is a projectivization of a $(n+1)$-dimensional vector space i.e. one has $\p^n = \p(\k^{n+1})$. Any automorphism $f$ of $\p^n$ is an equivalence class of linear operators $F\colon \k^{n+1}\to \k^{n+1}$ up to rescaling. Choose a basis $(e_0,e_1,\dots,e_{n})$ in the vector space $\k^{n+1}$ and denote by $(x_0:x_1:\dots:x_n)$ the homogeneous coordinates on $\p^n$ related to this basis. We say {\bf the automorphism $f$ of $\p^n$ is induced by the matrix~$M$ in coordinates $(x_0:x_1:\dots:x_n)$} if~$M$ is a matrix in the equivalence class of $F$ in the basis $(e_0,e_1,\dots,e_{n})$.

There exists a Jordan basis of the operator $M$ in $\k^{n+1}$. For simplicity we will always replace $M$ by the matrix in the equivalence class of $F$ such that the first eigenvalue of $M$ equals $1$. Blanc, in \cite{Blanc_conjugacy_of_PGL}, described the conjugacy classes of automorphisms of $\p^n$ under the action of the group $\Bir(\p^n)$. In order to formulate his result we need the following definition: the numbers $\lambda_1,\dots, \lambda_k\in\k$ are {\bf multiplicatively independent} if the only set of integers $m_1,\dots, m_k$ such that $\lambda_1^{m_1}\dots\lambda_k^{m_k} = 1$ is the set $m_1 = \dots=m_k = 0$. In the case when~\mbox{$\k = \CC$} numbers $\lambda_1,\dots, \lambda_k$ are multiplicatively independent if their logarithms are linearly independent over $\mathbb{Q}$.
\begin{theorem}[\cite{Blanc_conjugacy_of_PGL}]\label{theorem: conjugacy classes of PGL in Bir}
 Let $f$ be an infinite order element in $\Aut(\p^n)$. Then there exists a birational automorphism $\alpha\in \Bir(\p^n)$ such that $\alpha\circ f\circ\alpha^{-1}$ is also an element of $\Aut(\p^n)$ induced by a matrix of one of the following types:
  \begin{align}\label{eq: conjugacy classes of PGL in Bir}
   M_1 = \begin{pmatrix}
          1 & 0 & 0& \dots & 0\\ 0 &\lambda_1&0& \dots & 0 \\ 0&0&\lambda_2&\dots&0\\ \vdots &\vdots &\vdots& \ddots & \vdots\\ 0 & 0 &0& \dots &\lambda_n
         \end{pmatrix};
  &&
  M_2 = \begin{pmatrix}
          1 & 1 & 0& \dots & 0\\ 0 &1&0& \dots & 0 \\ 0&0&\mu_2&\dots&0\\ \vdots &\vdots &\vdots& \ddots & \vdots\\ 0 & 0 &0& \dots &\mu_{n}
         \end{pmatrix}.
  \end{align}
  Moreover, in the case of the matrix $M_1$ there exists $0\leqslant k\leqslant n-1$ such that numbers $\lambda_1,\dots, \lambda_k$ are roots of unity and $\lambda_{k+1},\dots, \lambda_n$ are multiplicatively independent in $\k$. Also for the matrix $M_2$ there exists $1\leqslant k\leqslant n$ such that numbers $\mu_2,\dots, \mu_k$ are roots of unity and $\mu_{k+1},\dots,\mu_n$ are multiplicatively independent in $\k$.
\end{theorem}
The proof of this theorem in \cite{Blanc_conjugacy_of_PGL} implies also the following assertion.
\begin{lemma}\label{lemma: exc and ind of blanc conjugation}
 Let $f$ be an element in $\Aut(\p^n)$ and let $\alpha\colon \p^n \dashrightarrow \p^n$ be the birational map described in Theorem \textup{\ref{theorem: conjugacy classes of PGL in Bir}}. Then $\Ind(\alpha)$ and $\Exc(\alpha)$ lie in the union of hyperplanes in $\p^n$.
\end{lemma}
\begin{proof}
 Since a regular automorphism of $\p^n$ preserves hyperplanes, we can assume that in homogeneous coordinates  $(x_0:\dots,x_n)$ the automorphism $f$ is induced by a Jordan matrix $M$. First, assume that $M$ is diagonal. By \cite[Proposition 6]{Blanc_conjugacy_of_PGL} there exists a birational map $\alpha\colon \p^n\dashrightarrow \p^n$ which conjugates the automorphism $f$  to an automorphism of $\p^n$ induced by the matrix $M_1$. Moreover, the map $\alpha$ is monomial i.e. $\alpha (x_0:\dots:x_n) = (Q_0:\dots:Q_n)$ where $Q_i\in \k[x_0,\dots,x_n]$ is a monomial for all $0\leqslant i\leqslant n$. Therefore, the sets $\Ind(\alpha)$ and $\Exc(\alpha)$ lie in the union of hyperplanes in $\p^n$.
 
 Now assume that $M$ contains a Jordan block of size at least $2$. Renumbering the basis, we can assume that $M$ is as follows:
 \[
  M = \begin{pmatrix}
       1 & 1 & 0&\dots &0\\ 0 & 1 & \eta&\dots & 0 \\
       0& 0& \lambda_3 &     \dots & 0\\
       \vdots & \vdots & \vdots & \ddots & \vdots\\ 0 & 0 & 0 & \dots & \lambda_n
      \end{pmatrix}
 \]
 Here either $\eta = 1$ if the first Jordan block is of size at least 3 or $\eta = 0$ otherwise. Numbers $1,1,\lambda_3,\dots,\lambda_n$ are eigenvalues of $M$. By \cite[Proposition 3]{Blanc_conjugacy_of_PGL} there exists a birational map $\beta\colon \p^n \dashrightarrow \p^n$ which conjugates~$f$ to a regular automorphism $\beta^{-1}\circ f\circ \beta$ of $\p^n$ induced by the following matrix:
 \[
  M' = \begin{pmatrix}
       1 & 1 & 0&\dots &0\\ 0 & 1 & 0&\dots & 0 \\
       0& 0& \lambda_3 &     \dots & 0\\
       \vdots & \vdots & \vdots & \ddots & \vdots\\ 0 & 0 & 0 & \dots & \lambda_n
      \end{pmatrix}
 \]
 Here the matrix $M'$ contains a unique Jordan block of size at least $2$. By the construction in the proof of~\mbox{ \cite[Lemma 1]{Blanc_conjugacy_of_PGL}} the birational map $\beta(x_0:x_1:\dots:x_n) = (P_0:P_1:\dots:P_n)$ is defined by the set of homogeneous polynomials $P_0,P_1,\dots,P_n$ such that $P_i\in \k[x_0,\dots, x_i]$ i.e. the $i$-th polynomial does not depend on coordinates $x_{i+1},\dots, x_n$. In particular, $P_0 = x_0^d$ for an integer $d\geqslant 1$.
 
 The indeterminacy locus of $\beta$ equals the zero locus of polynomials $P_0 = P_1 = \dots = P_n = 0$. Then $\Ind(\beta)$ lies in the hyperplane $\{x_0 = 0\}$. In order to show that the same is true for $\Exc(\beta)$ observe that $\beta$ induces a regular automorphism of the affine chart $U_0 = \{x_0\ne 0\}\cong \AA^n$ in $\p^n$. Therefore, $\Ind(\beta^{-1})$ also lies in $\{x_0 = 0\}$. Then the exceptional locus of $\beta$ lies in the set $\{P_0 = 0\} = \{x_0 = 0\}$. 
 
 Then by \cite[Proposition 6]{Blanc_conjugacy_of_PGL} there exists a birational map $\gamma\colon \p^n\dashrightarrow \p^n$ which conjugates the automorphism $\beta^{-1}\circ f\circ \beta$  to an automorphism of $\p^n$ induced by the matrix $M_2$. Denote by $Q_0,\dots,Q_n$ the polynomials in~\mbox{$\k[x_0,\dots, x_n]$} such that
 \[
  \gamma(x_0:\dots:x_n) = (Q_0:\dots:Q_n).
 \]
 By Blanc's construction $Q_0,\dots, Q_n$ are monomials. Moreover, we can assume that $Q_0 = x_0^d$ and $Q_1 = x_0^{d-1}x_1$ since the eigenvalue corresponding to the Jordan block does not change under the conjugation. Thus, $\gamma$ defines a regular automorphism of the affine chart $U_0\cong\AA^n$. Thus, $(\gamma\circ \beta)|_{U_0}$ is also an automorphism of $U_0$. Then $\alpha = \gamma\circ \beta$ and we conclude that $\Ind(\alpha)$ and $\Exc(\alpha)$ lie in the hyperplane $\{x_0 = 0\}$.
 \end{proof}

 \subsection{Eigenpolynomials under the action of $\PGL(n+1,\k)$} We  consider an automorphism $f$ of the projective space $\p^n$, fix homogeneous coordinates $(x_0:\dots:x_n)$ on $\p^n$ and choose a matrix $M$ inducing~$f$ with respect to these coordinates. The automorphism $f$ induces the action $f^*$ on the linear system~\mbox{$\p(H^0(\p^n,\OP(d)))$} for $d\geqslant 1$. Homogeneous coordinates induce the basis of the space $H^0(\p^n,\OP(d))$. Then the operator $f^*$ is induced by the matrix~$S^d M^t$ acting on~$H^0(\p^n,\OP(d))$ with respect to these coordinates. An invariant hypersubspace in the linear system corresponds to a subrepresentation of $H^0(\p^n,\OP(d))$ under the action of~$S^d M^t$. Note that this correspondence does not depend on the choice of the matrix $M$ in the equivalence class $F$. In this section we describe the subrepresentations in this space under the action of matrix $M_1$ and~$M_2$ defined in~\eqref{eq: conjugacy classes of PGL in Bir}. 
\begin{lemma}\label{lemma: M1-subrepresentations}
 Let $f$ be an automorphism of the projective space $\p^n$ induced by the matrix $M_1$ defined in~\textup{\eqref{eq: conjugacy classes of PGL in Bir}}. If $W\subset H^0(\p^n,\OP_{\p^n}(d))$ is a subrepresentation of $f^*$ then there exist polynomials $P_1,\dots, P_m\in\k[x_0,\dots, x_k]$ and monomials $Q_1,\dots, Q_m\in \k[x_{k+1},x_{k+2},\dots, x_n]$ such that:
 \[
  W = \langle P_1 Q_1 \dots, P_m Q_m \rangle.
 \]
\end{lemma}
\begin{proof}
 Denote by $V$ the vector space $H^0(\p^n,\OP(1))$, homogeneous coordinates $(x_0:...:x_n)$ define the basis of $V$. Since $V$ is dual to the underlying vector space of the projective space $\p^n$ then $f^*|_V$ is induced by the matrix $M = M_2^t$. Decompose $V$ into a direct sum of representations of $M$:
 \[
  V = U \oplus \bigoplus_{i = k+1}^n \k_{\lambda_{i}}.
 \]
 Here $U$ is generated by vectors $x_0,...,x_k$ and $\k_{\lambda_{i}}$ is a character generated by $x_i$. Since $H^0(\p^n, \OP(d)) = S^d V$ the action of $f^*$ on $H^0(X, \OP(d))$ is induced by the matrix $S^d M_1$. Then $S^dV$ decomposes into a sum of representations $S^d V = \bigoplus_I V_I$  where $I = (i_{k+1},\dots, i_n) \in Z_{\geqslant 0}^{n-k-1}$ is such that $i_{k+1}+\dots+ i_n = i$ and 
 \[
  V_I = S^{d-i}U \otimes K_{\lambda_{k+1}^{i_{k+1}}...\lambda_n^{i_n}}.
 \]
 Since  $\lambda_{k+1},...,\lambda_n$ are multiplicatively independent a subrepresentation of $S^d V$ is a sum $V_{I_1} + ... + V_{I_r}$. For any $I$ the representation $V_I$ is obviously generated by the necessary polynomials; thus, the proof is complete.
\end{proof}

In order to study subrepresentations of $M_2$ in $H^0(\p^n,\OP_{\p^n}(d))$ we need the following assertion.
\begin{lemma}\label{lemma: polynomial of x equals polynomial of x+1}
   Let $R$ be a commutative ring over $\k$ and let a polynomial $P(x)\in\k[x]$ and elements~$Q, r\in R$ be such that $P(x+r) = P(x) + Q$ and $r$ is not a zero divisor. Then there exists $Q'\in R$ such that $Q = r Q'$ and $P(x) = Q' x+\alpha$ where $\alpha\in R$.
 \end{lemma}
 \begin{proof}
  Assume that $a_0,\dots, a_{n-1}\in R$ and $a_n\in R\setminus\{0\}$ are such that $P(x) = \sum_{i=0}^n a_ix^i$, then we observe:
  \[
   P(x+r) = \sum_{i = 0}^n \left(\sum_{j = i}^n \binom{j}{j-i}a_jr^j\right)x^i.
  \]
  In particular, the coefficient of $x^n$ in $P(x+r)$ equals $a_n$. On the other hand, the coefficient of $x^n$ in the polynomial $P(x) + Q$ equals $1$ if $n\geqslant 1$ or $1+Q$ in the case $n=0$. If $n\geqslant 2$ the coefficient of $x^{n-1}$ in polynomials $P(x+r)$ and $P(x) + Q$ is equal to~$a_{n-1} + na_n r$ and $a_{n-1}$ respectively. Thus, $n$ is strictly less than 2.
  
  If $n = 1$ then the assumption implies $a_1 x + a_1r +a_0 = P(x+r) = P(x) + Q = a_1 x + a_0 + Q$. Thus, we get $Q' = a_1$ and $\alpha = a_0$. Finally, if $n=0$ then the assumption implies that $Q = 0$ and $\alpha = a_0$. 
 \end{proof}
 Lemma \ref{lemma: polynomial of x equals polynomial of x+1} implies the following description of irreducible subrepresentations of the automorphism induced by the matrix $M_2$ on the space $H^0(\p^n,\OP_{\p^n}(d))$.
 \begin{corollary}\label{corollary: M2-subrepresentations}
  Let $f$ be an automorphism of the projective space $\p^n$ induced by the matrix $M_2$ defined in~\textup{\eqref{eq: conjugacy classes of PGL in Bir}}. If $W\subset H^0(\p^n,\OP_{\p^n}(d))$ is an irreducible subrepresentation of $f^*$ then there exist a set of polynomials~\mbox{$P_0,\dots, P_m\in\k[x_0,x_2,\dots, x_k]$} and a monomial $Q\in \k[x_{k+1},x_{k+2},\dots, x_n]$ such that:
  \[
   W = \left\langle x_0^m P_0 Q, x_0^{m-1}(x_1P_0 + P_1)Q, \dots, \left(\sum_{i=0}^m x_1^{m-i}P_i\right)Q \right\rangle.
  \]
 \end{corollary}
 \begin{proof}
 By the same reason in the proof of Lemma \ref{lemma: M1-subrepresentations} we can reduce the question to the case when all eigenvalues of $M_2$ are roots of unity. Replacing $f$ by its iterate we can assume that $\mu_2 = ... = \mu_n = 1$.

  Consider a Jordan basis  $R_0, \dots, R_m \in  \k[x_0,\dots, x_k]$ of the operator $f^* = S^d M_2^t$ restricted to $W$. Since the representation $W$ is irreducible then the action of $f^* = S^d M_2^t$ is as follows:
  \[
   f^* R_i =\left\{\begin{aligned} 
                    &  R_0, && \text{ if $i = 0$};\\
                    &  R_i + R_{i-1}, && \text{ otherwise}.
                   \end{aligned}
      \right.
  \]
 By Lemma \ref{lemma: polynomial of x equals polynomial of x+1} we deduce that the polynomial $R_0$ lies in  $\k[x_0,x_2,\dots, x_n]$. Also we observe that for all~\mbox{$1\leqslant i\leqslant m$} the polynomial $R_i$ equals the following: 
 \[
  R_i = \frac{x_1R_{i-1}}{x_0} + P_i,
 \]
 where $P_i$ lie in $\k[x_0,x_2,\dots, x_n]$. Set $P_0 = R_0$ then we get the result.
 \end{proof}

 \section{Infinite order automorphisms}\label{section: results}
 \subsection{Invariant subvarieties of infinite order elements in $\Aut(\p^n)$} Here we describe the invariant subvarieties in the projective space $\p^n$ under the action of the automorphism induced by the matrix $M_1$ or~$M_2$ defined in \eqref{eq: conjugacy classes of PGL in Bir}.
 \begin{lemma}\label{lemma: invariant subvarieties under automorphisms of pn}
  Let $f\in \Aut(\p^n)$ be an automorphism induced by a matrix $M$ with respect to homogeneous coordinates $(x_0:x_1:\dots:x_n)$ and let $X$ be an $f$-invariant subvariety in $\p^n$. Assume that $X$ is irreducible and does not lie in a hyperplane $\{x_i = 0\}\subset \p^n$ for $0\leqslant i\leqslant n$. Then the following assertions hold.
  \begin{enumerate}
   \item[\textup{(1)}] Assume that $M = M_1$ defined in \textup{\eqref{eq: conjugacy classes of PGL in Bir}} and there exists $0\leqslant k\leqslant n-1$ such that the numbers $\lambda_1,\dots, \lambda_k$ are roots of unity and $\lambda_{k+1},\dots, \lambda_n$ are multiplicatively independent. Then $X$ is a cone whose vertex is the subspace $\{x_0 = \dots = x_k = 0\}$  over a subvariety in the subspace $\{ x_{k+1} = \dots = x_n = 0\}$.
   \item[\textup{(2)}] Assume that $M = M_2$ defined in \textup{\eqref{eq: conjugacy classes of PGL in Bir}} and there exists $1\leqslant k\leqslant n$ such that the numbers $\mu_2,\dots, \mu_k$ are roots of unity and $\mu_{k+1},\dots, \mu_n$ are multiplicatively independent. Then $X$ is a cone whose vertex is the subspace $\{x_0 = x_2 = \dots = x_k = 0\}$  over a subvariety in the subspace $\{ x_1 = x_{k+1} = \dots = x_n = 0\}$.
  \end{enumerate}
 \end{lemma}
 \begin{proof}
  There exists $m\geqslant 2$ such that $X$ equals the intersection of all hypersurfaces of degree $m$ containing $X$. Denote by $W\subset H^0(\p^n,\OP_{\p^n}(m))$ the subspace of these hypersurfaces i.e. $W \cong H^0(\p_n,\mathcal{I}_X(m))$ where~$\mathcal{I}_X$ is the sheaf of ideals of $X$. Then $W$ is an $f^*$-invariant subspace of $H^0(\p^n,\OP_{\p^n}(m))$. 
  
  If $M  = M_1$ then by Lemma \ref{lemma: M1-subrepresentations} the space $W$ is generated by the set of polynomials $P_1Q_1,\dots, P_mQ_m$ where $P_i\in\k[x_0,\dots, x_k]$ and $Q_i\in \k[x_{k+1},\dots,x_n]$ is a monomial for all $1\leqslant i\leqslant n$. Since $X$ does not lie in the hyperplane $\{x_i = 0\}\subset \p^n$ for $0\leqslant i\leqslant n$ we deduce that $X$ equals the following:
  \[
   X = \{P_1 = \dots = P_m = 0\}.
  \]
 Since this set is a cone whose vertex is the subspace $\{x_0 = x_2 = \dots = x_k = 0\}$  over a hyperplane in the subspace $\{ x_1 = x_{k+1} = \dots = x_n = 0\}$ we get the result.
  
  If $M = M_2$ then denote by $W_1,\dots, W_r$ the irreducible subrepresentations of $W$ i.e. $W = \bigoplus_{i=1}^rW_i$. Denote by $X_i$ the following subvariety in $\p^n$:
  \[
   X_i = \bigcap_{P\in W_i} \{P = 0\}.
  \]
  Then $X$ equals the intersection of $X_1,\dots, X_r$. If for all $1\leqslant i\leqslant r$ the variety $X_i$ is a cone with a vertex in  the subspace $\{ x_1 = x_{k+1} = \dots = x_n = 0\}$ then so is $X$. Thus, it suffices to prove the assertion in the case when $W$ is irreducible.
  
  If $W$ is irreducible then Corollary \ref{corollary: M2-subrepresentations} provides the description of a basis of $W$. Since $X$ does not lie in a hyperplane $\{x_i = 0\}\subset \p^n$ for $0\leqslant i\leqslant n$ we deduce that
  \[
   X = \{P_0 = \dots = P_m = 0\},
  \]
  in notation of Corollary \ref{corollary: M2-subrepresentations}. Thus, $X$ is a cone with a vertex in  the subspace $\{ x_0 = x_2 = \dots = x_k = 0\}$ and the proof is complete.
 \end{proof}

 \subsection{Action of automorphisms on $\Pic^0(X)$ and $\Alb(X)$}
  Recall that the {\bf Picard group $\Pic(X)$} is the algebraic group of  line bundles on $X$. By $\Pic^0(X)$ we denote connected component of $\Pic(X)$ containing the structure sheaf $\OP_X$. The quotient group $\Pic(X)/\Pic^0(X)$ is called the {\bf Neron--Severi group $NS(X)$ of~$X$}. If $X$ is smooth and projective then the rank of the discrete abelian group $\NS(X)$ is finite. 
  
  If $X$ is a smooth projective and complex variety then the algebraic group $\Pic^0(X)$ can be identified with the abelian variety $H^1(X,\OP_X)/H^2(X,\ZZ)$, see, for instance, \cite[Section 7.2.2]{Voisin_book}. In view of this we can show that the inverse image of an automorphism of $X$ whose growth is bounded induces finite order operators on $\Pic^0(X)$ and $\NS(X)$.
 \begin{lemma}\label{lemma: bounded growth implies finite order on Pic0 and NS}
  Let $X$ be a smooth projective complex variety and let $f\in \Aut(X)$ be its automorphism whose growth is bounded. Then there exists $N\geqslant 0$ such that the inverse images
  \begin{align*}
   f^*_{\NS}\colon \NS(X)\to \NS(X); && f^*_{\Pic^0}\colon \Pic^0(X)\to \Pic^0(X)
  \end{align*}
   are identities.
 \end{lemma}
 \begin{proof}
  First consider the inverse image $f^*_{\NS}\colon \NS(X)\to \NS(X)$. The rank of the abelian group $\NS(X)$ is finite; thus, it suffices to prove that the order of the operator $f^*_{\NS_{\RR}}\colon \NS_{\RR}(X)\to \NS_{\RR}(X)$ is finite, here by~\mbox{$\NS_{\RR}(X)$} we denote the~\mbox{$\RR$-vec}\-tor space  $\NS(X)\otimes_{\ZZ} \RR$. Then either $f^*_{\NS_{\RR}}$ is diagoizable over $\CC$ and all its eigenvalues are roots of unity or the growth of $f$ is not bounded. Thus, the order of $f^*_{\NS_{\RR}}$ is finite.
  
  The inverse image~$f^*_{\Pic^0}$ preserve a point on $\Pic^0(X)$. Indeed, $f^*\OP_X = \OP_X$ in the Picard group of $X$. Since $\Pic^0(X)$ is an abelian variety whose universal cover is canonically isomorphic to  $H^1(X,\OP_X)$. Thus, it suffices to show that the order of the following operator is finite:
  \[
  f_{H^{1,0}}^*\colon H^1(X,\OP_X)\to H^1(X,\OP_X).
  \]
  By Hodge decomposition one has $H^1(X,\CC)\cong H^{1,0}(X) \oplus H^{0,1}$ and $H^1(X,\OP_X)\cong H^{1,0}(X)$. Since the inverse image map $f^*_{H^1(X,\CC)}\colon H^1(X,\CC)\to H^1(X,\CC)$ preserves the lattice $H^1(X,\ZZ)$ all eigenvalues of $f_{H^{1,0}}^*$ are algebraic integers. 
  
  Assume there is an eigenvalue of $f_{H^{1,0}}$ that equals $\lambda$  and it is not a root of unity; thus, $|\lambda|\ne 1$. Then there is a non-zero eigenclass $\xi\in H^{1,0}(X)$ such that $f^*\xi = \lambda \xi$. Consider the class $\xi\wedge \overline{\xi}\in H^{1,1}(X)$. This class is non-zero by construction and
  \[
  f^*(\xi\wedge \overline{\xi}) = |\lambda|^2\ne 1.
  \]
  Thus, there is an eigenvalue of $f^*_{H^{1,1}(X)}\colon {H^{1,1}(X)}\to {H^{1,1}(X)}$ whose absolute value is not equal to 1. This contradicts the assumption that the growth of $f$ is bounded. Thus, all eigenvalues of $f_{H^{1,0}}^*$ are roots of unity. By the same reason $f_{H^{1,0}}^*$ has no non-trivial Jordan blocks. Thus, the proof is complete.
 \end{proof}
 The next technical assertion will be necessary for the study of automorphisms inducing identity on the Albanese variety.
 \begin{lemma}\label{lemma: bounded order automorphisms of projective bundles}
  Let $T$ be a smooth projective complex variety, let $F\in \PGL(n+1, \CC(T))$ be a birational automorphism  of the direct product $\p^n\times T$ and let $X$ be an $F$-invariant irreducible subvariety of $\p^n\times T$ such that $\pr_2(X) = T$ where $\pr_2\colon \p^n\times T\to T$ is the projection to the second component of the product. If~$X_{\CC(T)}$ does not coincide with a hypersubspace in $\p^n_{\CC(T)}$ and the growth of $F|_{X}$ is bounded, then there exists a birational map $\alpha\colon X\dashrightarrow Y$ and an ample line bundle $\L\in\Pic(Y)$ such that $g = \alpha\circ f\circ \alpha^{-1}$ is a regular automorphism of $Y$ and $g^*\L = \L$.
 \end{lemma}
 \begin{proof}
 Consider the characteristic polynomial $\chi$ of matrix $M$ in the equivalence class of $F$ with respect to homogeneous coordinates $(x_0:\dots:x_n)$. Then $\chi$ is a polynomial over the field $\CC(T)$ and choosing a good representative $M$ we can suppose that 1 is a root of $\chi$. Denote by $K$ the splitting field of $\chi$.
 
  First we assume that the field extension $\CC\subset K$ is finite. Then $\CC = K$ since $\CC$ is algebraically closed. Therefore,~$F$ is a regular automorphism $f\times \mathrm{Id}_T\in \Aut(\p^n\times T)$. We choose an ample line bundle $\L'$ in~\mbox{$\Pic^0(\p^n\times T)$} and set $Y = X$ and $\L$ is the inverse image of $\L'$ to $X$.

  Now assume that the field extension $\CC\subset K$ is infinite. Fix homogeneous coordinates  $(x_0:\dots:x_n)$ associated to a Jordan basis of $F$ over the field $K$. By Theorem \ref{theorem: conjugacy classes of PGL in Bir} and Lemma \ref{lemma: exc and ind of blanc conjugation} we can assume that $F$ is induced by matrix $M_1$ or $M_2$ defined in \eqref{eq: conjugacy classes of PGL in Bir}. Since at least one eigenvalue of $F$ is outside $\CC$ we observe that there exists $k\leqslant n-1$ such that~$\lambda_0,\dots, \lambda_k$ are roots of unity and $\lambda_{k+1},\dots, \lambda_n$ are multiplicatively independent. Fix elements $a_0,\dots, a_n\in K^*$ such that the function $L(x) = \sum_{i=1}^n a_i x_i$ defines a non-empty hyperplane in $\p^n_{\CC}$. 
  
  Let $\xi$ be the class of the divisor $D= \{L(x) = 0\}\subset \p^n\times T$ in the Neron--Severi group $\NS(\p^n\times T)$. Fix an integer $N\geqslant 1$ and define the divisor $D_N =  \{L(F^N(x)) = 0\}\subset \p^n\times T$ where $L(\lambda^Nx)$ is the following function:
  \[
   L(\lambda^N x) =\left\{ \begin{aligned}
                    & \sum_{i=0}^n a_i \lambda_i^N x_i, &&\text{ if $F$ is induced by the matrix $M_1;$}\\
                    & (a_0 + N a_1) x_0 + \sum_{i=1}^n a_i \lambda_i^N x_i, &&\text{ if $F$ is induced by the matrix $M_2.$}
                    \end{aligned}
    \right.
  \]
  Since $D_N$ is the proper preimage of $D$ under $F^N$ then $(F^N)^*(\xi) \geqslant [D_N]$. We are going to show that there exists a curve $C$ in $X$ such that $[D_N] \cdot[C]$ tends to infinity. Thus, $(F^N)^*(\xi)\cdot [C]$ also tends to infinity which contradicts to the assumption that $F |_X$ is bounded.

  Fix a curve $C\subset X$ such that $C' = \pr_2(C)$ is a curve on $T$ such that $\lambda_n|_{C'}$ is not constant. The rational function $L(\lambda^Nx)$ is not constant on $C$ for any $N\geqslant 0$ since $a_n\ne 0$.
  Then the intersection $D_N\cap C$ is finite and it contains the points $((0:\dots:0:1), t)$ where $t\in C$ is a zero of $(\lambda_n)^N|_{C}$. Since $\lambda_n|_{C'}$ is not constant the sequence of intersection numbers $(f^N)^*(\xi)\cdot [C]$ is not bounded and the proof is complete.
 \end{proof}
 Recall that given a smooth projective variety $X$ there exists an abelian variety $ \Alb(X)$ which is called the {\bf Albanese variety} of $X$ and a morphism $a\colon X\to  \Alb(X)$ which is universal i.e. for any algebraic morphism $b\colon X\to B$ to an abelian variety $B$ there exists a unique morphism of varieties $\varphi \colon  \Alb(X)\to B$ such that the following diagram commutes:
\[\begin{tikzcd}[ampersand replacement=\&]
	X \\
	 \Alb(X) \& B
	\arrow["a"', from=1-1, to=2-1]
	\arrow["b", from=1-1, to=2-2]
	\arrow["\varphi"', from=2-1, to=2-2]
\end{tikzcd}\]
 In particular, any automorphism of $X$ factors through the Albanese variety i.e. for any $f\in \Aut(X)$ there exists $f_a\in \Aut(\Alb(X))$ such that $a\circ f = f_a \circ a$. Now we are ready to show that an automorphism whose growth is bounded either induces an infinite order automorphism of the Albanese variety or preserves a big and nef line bundle.
 \begin{corollary}\label{corollary: either translation on albanese or invariant big nef bundle}
  Let $X$ be an algebraic variety over $\CC$ and let $f\in \Aut(X)$ be its automorphism whose growth is bounded. Then there exists $N\geqslant 1$ such that the induced automorphism $f_a^N\colon \Alb(X)\to \Alb(X)$ is a translation. Moreover, if $f_a^N = \mathrm{Id}_{\Alb(X)}$ then there exists a big and nef line bundle $\L\in \Pic(X)$ such that $f^*\L = \L$.
 \end{corollary}
 \begin{proof}
  By Lemma \ref{lemma: bounded growth implies finite order on Pic0 and NS} we deduce that there exists $N\geqslant 1$ such that $(f^N)^*_{\Pic^0} = \mathrm{Id}_{\Pic^0(X)}$. Since the abelian varieties $\Pic^0(X)$ and $\Alb(X)$ are dual this implies that $f_a^N$ is a translation.
  
  Now assume that $(f^N)^*_{\NS} = \mathrm{Id}_{\NS(X)}$ and $f_a^N = \mathrm{Id}_{\Alb(X)}$. For simplicity further we  use the notation~\mbox{$A = \Alb(X)$}.  Fix a very ample line bundle $\M\in \Pic(X)$, then there exists a line bundle $\F\in \Pic^0(X)$ such that $f^*\M = \M\otimes \F$. Consider the Albanese morphism $a\colon X\to A$, by construction one has the following isomorphism:
  \[
   a^*\colon \Pic^0(A)\to \Pic^0(X).
  \]
  Thus, there exists a line bundle $\F'\in\Pic^0(A)$ such that $\F = a^*\F'$. Fix a Zariski open subset $U\subset A$ such that $\F'|_{U} = \OP_U$. Denote by $X_U$ the preimage $X_U = a^{-1}(U)$. Then one has $(f^N)^*|_{X_U}(\M|_{X_U}) = \M|_{X_U}$. The linear system of $\M$ induces the embedding $\varphi_U\colon X_U \to \p^n\times U$ and the automorphism $f^N$ induces an automorphism $F_U$ of the product $\p^n\times U$. Since $f$ induces identity on $A$ then the automorphism $F_U$ is an element in the group $\PGL(n+1, \CC[U])$ and $X_U$ is an $F_U$-invariant subvariety.
  
  Denote by $\widetilde{X}$ the proper image of $X_U$ under the embedding $\p^n\times U\hookrightarrow \p^n\times A$. The automorphism~$F_U$ induces a birational automorphism $\widetilde{F}$ of $\p^n\times U$ such that $\widetilde{X}$ is $\widetilde{F}$-invariant and the growth of $\widetilde{F}|_{\widetilde{X}}$ is bounded. Note that $\widetilde{X}_{\CC(A)}$ is not a hyperplane in $\p^n_{\CC(A)}$ since the embedding is defined by the linear system on $\widetilde{X}_{\CC(A)}$. Then by Lemma \ref{lemma: bounded order automorphisms of projective bundles} there exists an ample line bundle $\widetilde{\L}$ on $\widetilde{X}$ such that 
  \[
   (\widetilde{F}|_{\widetilde{X}})^*{\widetilde{\L}} = \widetilde{\L}.
  \]
  Denote by $\alpha\colon X\to \widetilde{X}$ the birational map induced by the embedding $\p^n\times U\hookrightarrow \p^n\times A$. The inverse image~$\L' = \alpha^*\widetilde{\L}$ is a big and nef line bundle and one has $(f^N)^*\L' = \L'$. Set
  \[
   \L = \L'\otimes f^*\L'\otimes \dots\otimes (f^{(N-1)})^*\L'.
  \]
  Then $\L$ is the big and nef line bundle with the property $f^*\L = \L$. The proof is complete.
 \end{proof}

 \subsection{Proofs}
 In this section we complete proofs of Theorem \ref{theorem: automorphism of bounded growth implies ruled structure or translation on albanese} and Corollary \ref{corollary: generalization of Blanc-Deserti}. We will need the following assertion describing the bounded order automorphisms which preserve a big and nef line bundle.
 \begin{lemma}\label{lemma: proof of thm in the case of invariant big nef line bundle}
  Let $X$ be a smooth projective variety over an algebraically closed field of characteristic $0$  and let $f\in \Aut(X)$ be an automorphism whose growth is bounded and the order is infinite. If there exists a big and nef line bundle $\L\in \Pic(X)$ such that $f^*\L = \L$ then we get the case $(3)$ of Theorem \textup{\ref{theorem: automorphism of bounded growth implies ruled structure or translation on albanese}}. 
 \end{lemma}
 \begin{proof}
  The linear system of $\L$ induces the following rational map:
 \[
  \varphi_{|\L|}\colon X\dashrightarrow \p^n.
 \]
 The map $\varphi_{|\L|}$ maps $X$ birationally to its proper image in $\p^n$ and $f$ induces a regular automorphism $F$ of~$\p^n$. We replace $X$ with its proper image in $\p^n$.
 
 By construction no hyperplane in the projective space $\p^n$ contains the variety~$X$. By Theorem \ref{theorem: conjugacy classes of PGL in Bir} there exists a birational map~\mbox{$\alpha\colon \p^n\dashrightarrow \p^n$} such that $G = \alpha\circ F\circ \alpha^{-1}$ is induced by the matrix $M_1$ or $M_2$ defined in \eqref{eq: conjugacy classes of PGL in Bir}. Moreover, by Lemma \ref{lemma: exc and ind of blanc conjugation} the restriction $\alpha|_{X}\colon X\dashrightarrow Y$ is a birational map to the proper image $Y$ of $X$. 
 
 Then $Y$ is a $G$-invariant subvariety in $\p^n$ and $G$ is an infinite order automorphism of $\p^n$. We consider the maximal intersection of hyperplanes $\Pi = \{x_{i_1} = x_{i_2} = \dots = x_{i_k} = 0\}$ which contains $Y$ and replace~$\p^n$ with this intersection and $G$ with $G|_{\Pi}$. Then $Y\subset \p^n$ is a $G$-invariant subvariety that does not lie in a hyperplane~\mbox{$\{x_i = 0\}$} and $G$ is induced by one of the matrices $M_1$ or $M_2$. 
 
 By Lemma \ref{lemma: invariant subvarieties under automorphisms of pn} we deduce that $Y$ is a cone. More precisely, in the case when $G$ is induced by the matrix $M_1$ we observe the variety $Y$ is a cone with a vertex $\{x_{k+1} = \dots = x_n = 0\}$ over a subvariety $Z$ in the hypersubspace $H = \{x_{0} = \dots = x_k = 0\}$ for~\mbox{$0\leqslant k\leqslant n-1$}. Moreover, the restriction of $G$ to the hypersubspace $H$ is an automorphism of finite order. Thus, the automorphism $G|_{Y}$ is conjugate to the product $g\times h\in \Aut(\p^{n-k})\times \Aut(Z)$ where the order of $h$ is finite and $g$ is an automorphism of $\p^{k+1}$ induced by the following matrix:
 \[
  M_g = \mathrm{diag}(1, \lambda_{k+1},\dots, \lambda_n),
 \]
 where numbers $\lambda_{k+1},\dots, \lambda_n$ are as in the matrix $M_1$. In particular, they are multiplicatively independent. The case when $G$ is induced by the matrix $M_2$ is analogous. This completes the proof.
 \end{proof}
 Now we are ready to prove Theorem \ref{theorem: automorphism of bounded growth implies ruled structure or translation on albanese}.
 \begin{proof}[Proof of Theorem \textup{\ref{theorem: automorphism of bounded growth implies ruled structure or translation on albanese}}]
 Let $f\in \Aut(X)$ be an automorphism whose growth is bounded. Once the order of~$f$ is finite, then we get the case \textup{(1)} so further we assume that the order of $f$ is infinite. By Corollary~\ref{corollary: either translation on albanese or invariant big nef bundle} either there exists $N$ such that $f^N$ induces an infinite order translation on the Albanese variety of $X$ i.e. we get the case \textup{(2)} or there exists a big and nef line bundle $\L$ such that $f^*\L = \L$. Then by Lemma \ref{lemma: proof of thm in the case of invariant big nef line bundle} we get the case \textup{(3)}.
 \end{proof}
 Here is the proof of Corollary \textup{\ref{corollary: generalization of Blanc-Deserti}}.
 \begin{proof}[Proof of Corollary \textup{\ref{corollary: generalization of Blanc-Deserti}}]
 Consider a rationally connected threefold $X$ and its infinite order automorphism $f$ whose growth is bounded. Then by Lemma \ref{lemma: proof of thm in the case of invariant big nef line bundle} there exists a birational isomorphism $\alpha\colon X\dashrightarrow \p^m\times Z$ where~\mbox{$m\geqslant 1$} and $\alpha\circ f\circ \alpha^{-1}=g\times h\in \Aut(\p^m)\times \Aut(Z)$ where $g$ is an infinite order automorphism of $\p^m$ and the order of $h$ is finite. Let $N$ be the order of $h$, then $\alpha\circ f^N\circ \alpha^{-1} = g^N\times \mathrm{id}_Z$.
 
 Since $X$ is rationally connected, so is $Z$. Since $\dim(Z)\leqslant 2$ then $Z$ is a rational variety (or a point). Thus, there exists a birational map $\beta \colon Z\dashrightarrow \p^{3-m}$. Denote by $\alpha'$ the composition 
 \[
  \alpha' = \left(\mathrm{id}_{\p^m}\times \beta\right)\circ \alpha\colon X\dashrightarrow \p^k\times \p^{3-m}.
 \]
 This map is birational by construction; moreover, $\alpha'\circ f^N \circ {\alpha'}^{-1} = g\times \mathrm{id}_{\p^{3-m}}\in \Aut(\p^m)\times \Aut(\p^{3-m})$. This automorphism is conjugate to an automorphism of $\p^3$ and the proof is complete. 
 \end{proof}
 \section{Automorphism of a rational threefold which is not conjugate to $\Aut(\p^3)$}\label{section: example} 
 Let $S$ be a del Pezzo surface of degree $2$ and let $\pi\colon S\to \p^2$ be the map induced by the anticanonical linear system. Then $\pi$ is a double cover branched over a smooth quartic curve. Denote by $\tau\in \Aut(S)$ the Geiser involution of $S$. Fix an element $g\in\Aut(\p^1)$ of infinite order and consider the automorphism $\tau\times g$ of the rational threefold $S\times \p^1$. 
 
 The automorphism $\tau \times g$ is not conjugate to a regular automorphism of $\p^3$ by Lemma \ref{lemma: example of threefold aut which is not conjugate to pgl}. The proof follows from the fact that the involution $\tau$ is not conjugate to a regular automorphism of $\p^2$. Thus, we get a construction of an infinite order automorphism of a rational threefold whose growth is bounded and which is not conjugate to a regular automorphism of $\p^3$.
 \begin{lemma}\label{lemma: example of threefold aut which is not conjugate to pgl}
  The automorphism $\tau\times g\in \Aut(S\times \p^1)$ is not conjugate to a regular automorphism of $\p^3$.
 \end{lemma}
 \begin{proof}
  Denote by $Y$ the product $Y = R\times \p^1$ in $ S\times \p^1$, here $R\subset S$ is the ramification divisor of the double cover $\pi\colon S\to \p^2$. Thus, $R$ is a smooth curve of genus $3$ and the divisor $Y$ is $(\tau\times g)$-invariant. 
  
  Fix a $g$-invariant point $x_0\in\p^1$ and consider the curve $C = R\times \{x_0\}\subset S\times \p^1$. The curve $C\subset Y$ is smooth of genus $3$ and it is $(\tau\times g)$-invariant. Moreover, one has $(\tau\times g)|_{C} = \mathrm{id}_C$.
  
  Assume that there exists a birational map $\alpha\colon S\times \p^1\dashrightarrow \p^3$ such that $f = \alpha\circ (\tau\times g)\circ \alpha^{-1} \in \Aut(\p^3)$. The curve $C$ does not lie in the indeterminacy locus of $\alpha$ since it is not rational. Consider the proper images~$Y'$ and $C'$ of the divisor $Y$ and the curve $C$ under the map $\alpha$. Then $Y'$ is either an  $f$-invariant surface birational to $Y$ or it is an  $f$-invariant curve of geometric genus $3$ and $C'$ is an $f$-invariant curve in~$\p^3$. By construction we have $f|_{C'} = \mathrm{id}_C'$.
  
  If $Y'$ is a surface then the restriction of $\alpha$ to $C$ is a birational map. Thus, $C'$ is an $f$-invariant curve in~$\p^3$ of geometric genus 3. Since $f|_{C'} = \mathrm{id}_C'$ we deduce that $C'$ lies in a hyperplane of $\p^3$. Moreover, there exists homogeneous coordinates $(x_0:x_1:x_2:x_3)$ of $\p^3$ such that $C'\subset\{x_3 = 0\}$ and
  \[
   f(x_0:x_1:x_2:x_3) = (x_0:x_1:x_2:\lambda x_3).
  \]
  Consider the divisor $S_0 = S\times \{x_0\} \subset S\times \p^1$. Then $S_0$ is $(\tau\times g)$-invariant divisor and
  \[
   (\tau\times g)|_{S_0}\colon S_0\to S_0 
  \]
  is an involution. Consider the proper image $S'_0$ of the divisor $S_0$. Since $C$ lies in $S_0$ and $\alpha|_C$ is a birational map then $S'_0$ is a divisor in $\p^3$. By construction it is $f$-invariant and $f|_{S_0'}$ is an involution. Then $S_0'$ does not lie in $\{x_3 = 0\}$; thus, $\lambda = 1$ i.e. the order of automorphism $f$ is finite. This contradicts the assumption.
  
  If $Y'$ is a curve then $\alpha|_{Y}$ is the projection to the first component of the product $C\times \p^1$. Thus, the restriction of $\alpha$ to $C$ is birational and by the above argument we also get a contradiction with the fact that the order of $f$ is infinite. This finishes the proof.
 \end{proof}

   \bibliographystyle{alpha}
   \bibliography{references.bib}

\end{document}